\documentclass[12pt]{amsart}
\usepackage{amssymb}

\def\mod#1{{\ifmmode\text{\rm\ (mod~$#1$)}
\else\discretionary{}{}{\hbox{ }}\rm(mod~$#1$)\fi}}

\usepackage{float}
\restylefloat{table}

\newtheorem{theorem}{Theorem}[section]
\newtheorem{lemma}[theorem]{Lemma}

\newtheorem{proposition}[theorem]{Proposition}
\newtheorem{conjecture}[theorem]{Conjecture}
\theoremstyle{remark}
\newtheorem{rem}{Remark}
\newtheorem{dfn}{Definition}

\numberwithin{equation}{section}

\def\mod#1{{\ifmmode\text{\rm\ (mod~$#1$)}
\else\discretionary{}{}{\hbox{ }}\rm(mod~$#1$)\fi}}

\usepackage{hyperref}

\begin{document}
\title{Mordell's equation : a classical approach}

\author{Michael A. Bennett}
\address{Department of Mathematics, University of British Columbia, Vancouver, BC Canada V6T 1Z2}
\email{bennett@math.ubc.ca}
\thanks{Supported in part by a grant from NSERC}

\author{Amir Ghadermarzi}
\address{Department of Mathematics, University of British Columbia, Vancouver, BC Canada V6T 1Z2}
\email{amir@math.ubc.ca}

\date{\today}

\begin{abstract}
We solve the Diophantine equation $Y^2=X^3+k$ for all nonzero integers $k$ with $|k| \leq 10^7$.
Our approach uses a classical connection between these equations and cubic Thue equations. The latter  can be treated algorithmically via lower bounds for linear forms in logarithms in conjunction with lattice-basis reduction.
\end{abstract}

\maketitle

\section{Introduction}
If $k$ is a nonzero integer, then the equation  
\begin{equation} \label{Mord}
Y^2=X^3+k
\end{equation}
defines an elliptic curve over $\mathbb{Q}$. Such Diophantine equations have a long history, dating  back (at least) to work of Bachet in the 17th century, and are nowadays termed {\it Mordell equations}, honouring the  substantial contributions of L. J. Mordell to their study. Indeed the statement that, for a given $k \neq 0$, equation (\ref{Mord}) has at most  finitely many integral solutions is implicit in work of Mordell \cite{Mor2} (via application of a result of Thue \cite{Th}), and explicitly stated in \cite{Mor3}.

Historically, the earliest approaches to equation (\ref{Mord}) for certain special values of $k$ appealed to simple local arguments; references to such work may be found in Dickson \cite{Dic}. More generally, working in either $\mathbb{Q} (\sqrt{-k})$ or $\mathbb{Q} (\sqrt[3]{k})$, one is led to consider a finite number of Thue equations of the shape $F(x,y)=m$, where the $m$ are nonzero integers and the $F$ are, respectively, binary cubic or quartic forms with rational integer coefficients.  Via classical arguments of Lagrange (see e.g. page 673 of Dickson \cite{Dic}), these in turn  correspond to a finite (though typically larger) collection of Thue equations of the shape $G(x,y)=1$. Here, again, the $G$ are binary cubic or quartic forms with integer coefficients. In case $k$ is positive, one encounters cubic forms of negative discriminant which may typically be treated rather easily via Skolem's $p$-adic method (as the corresponding cubic fields have a single fundamental unit). For negative values of $k$, one is led to cubic or quartic fields with a pair of fundamental units, which may sometimes be treated by similar if rather more complicated methods; see e.g. \cite{FL},  \cite{He} and \cite{Lj}.

There are alternative approaches for finding the integral points on a given model of an elliptic curve. The most commonly used currently proceeds via appeal to lower bounds for linear forms in elliptic logarithms,  the idea for which dates back to work of Lang \cite{L} and Zagier \cite{Z} (though the  bounds required to make such arguments explicit are found in work of David and of Hirata-Kohno, see e.g. \cite{DaHi}). Using these bounds, Gebel, Peth\H{o} and Zimmer \cite{G}, Smart \cite{S} and Stroeker, Tzanakis \cite{T} obtained, independently,  a ``practical'' method to find integral points on elliptic curves. Applying this method, in 1998, Gebel, Peth\H{o}, Zimmer \cite{P} solved equation (\ref{Mord}) for all integers $\left\lvert k \right\rvert <10^4$ and partially extended the computation to 
$\left\lvert k \right\rvert <10^5$. As a byproduct of their calculation, they obtained a variety of interesting  information about the corresponding elliptic curves, such as their ranks, generators of their Mordell-Weil groups, and information on their Tate-Shafarevic  groups. 

The only obvious disadvantage of this approach is its dependence upon knowledge of the Mordell-Weil basis over $\mathbb{Q}$ of the given Mordell curve (indeed, it is this dependence that ensures, with current technology at least, that this method is not strictly speaking algorithmic). For curves of large rank, in practical terms, this means that the method cannot be guaranteed to solve equation (\ref{Mord}). 

Our goal in this paper is to present (and demonstrate the results of) a practical algorithm for solving Mordell equations with values of $k$ in a somewhat larger range. At its heart are lower bounds for linear forms in complex logarithms, stemming from the work of Baker \cite{Bak1}. These were first applied in the context of explicitly solving equation (\ref{Mord}), for fixed $k$, by Ellison et al \cite{Ell}.
To handle values of $k$ in a relatively large range, we will appeal to classical invariant theory and, in particular, to the reduction theory of binary cubic forms (where we have available very accessible, algorithmic work of Belabas \cite{Be}, Belabas and Cohen \cite{BeCo} and Cremona \cite{Cr}). This approach has previously been outlined by Delone and Fadeev  (see Section 78 of \cite{DF}) and  Mordell (see e.g. \cite{Mor}; its origins lie in \cite{Mor1}). We use it to solve Mordell's equation (\ref{Mord}) for all $k$ with $0< \left \lvert k \right\rvert \leq 10^7$. We should emphasize that our algorithm does not provide {\it a priori} information upon, say, the ranks of the corresponding elliptic curves, though values of $k$ for which (\ref{Mord}) has many solutions necessarily (as long as $k$ is $6$-th power free) provide curves with at least moderately large rank (see \cite{GrSi}).

We proceed as follows. In Section  \ref{Pre}, we will discuss the precise correspondence that exists between integer solutions to (\ref{Mord}) and integer solutions to cubic Thue equations of the shape $F(x,y)=1$, for certain binary cubic forms of discriminant $-108k$. In Section \ref{rep}, we indicate a method to choose representatives from equivalence classes of forms of a given discriminant. Section \ref{comp} contains a brief discussion of our computation, while Section \ref{Numer} is devoted to presenting a summary of our data, including information on both the number of solutions, and on their heights.

\section{Preliminaries} \label{Pre}

In this section, we begin by outlining a correspondence between integer solutions to the equation $Y^2=X^3+k$ and solutions to certain cubic Thue equations of the form $F(x,y)=1$, where $F$ is a binary cubic form of discriminant $-108k$. As noted earlier, this approach is very classical. To make it computationally efficient, however, there are a number of details that we must treat rather carefully.

Let us suppose that $a, b, c$ and $d$ are integers, and consider the binary cubic form 
\begin{equation} \label{form}
F(x,y)=ax^3+3bx^2y+3cxy^2+dy^3, 
\end{equation}
with discriminant 
$$
D=D_{F}=-27 \left( a^2d^2-6abcd-3b^2c^2+4ac^3+4b^3d \right).
$$
It is important to observe (as a short calculation reveals) that the set of forms of the shape (\ref{form}) is closed within the larger set of binary cubic forms in $\mathbb{Z}[x,y]$, under the action of both $SL_2 ( \mathbb{Z})$ and $GL_2 ( \mathbb{Z})$.
To such a form we associate  covariants, namely  the Hessian  $H=H_F (x,y) $ given by
$$
H=  H_F (x,y)=  - \frac{1}{4} \left(\frac{\partial^2 F}{\partial x^2} \frac{\partial^2 F}{\partial y^2} - \left(\frac{\partial^2 F}{\partial x \partial y}\right)^2 \right) 
$$  
 and the Jacobian determinant of $F$ and $H$, a cubic form $G=G_F$ defined as
$$
G=G_F (x,y)=\frac{\partial F}{\partial x}\frac{\partial H}{\partial y}-  \frac{\partial F}{\partial y} \frac{\partial H}{\partial x}.
$$
Note that, explicitly, we have
$$
H/9= (b^2-ac) x^2 + (bc-ad) xy + (c^2-bd) y^2 
$$
and
$$
G/27=  a_1 x^3 + 3 b_1 x^2 y + 3 c_1 x y^2 + d_1 y^3,
$$
where
$$
a_1=-a^2d+3abc-2b^3, \; b_1=-b^2c-abd+2ac^2, \; c_1=bc^2-2b^2d+acd
$$
and $d_1=-3bcd+2c^3+ad^2$.

Crucially for our arguments, these covariants satisfy the syzygy
    $$
     4H(x,y)^3=G(x,y)^2+27D F(x,y)^2.
  $$
Defining $D_1=D/27$, $H_1=H/9$ and $G_1=G/27$, we thus have
$$
4H_1(x,y)^3=G_1(x,y)^2+D_1 F(x,y)^2.
$$ 
If $(x_0,y_0)$ satisfies the equation $F(x_0,y_0)=1$ and $D_1 \equiv 0 \mod{4}$ (i.e. if $ad \equiv bc \mod{2}$), then necessarily $G_1(x_0,y_0) \equiv 0 \mod {2}$. We may therefore conclude that $Y^2=X^3+k$, where
$$
X=H_1(x_0,y_0), \; \; Y=\frac{G_1(x_0,y_0)}{2} \; \mbox{ and } \; k = -\frac{D_1}{4} =  -\frac{D}{108}.
$$
   
 It follows that, to a given triple $(F,x_0,y_0)$, where $F$ is a cubic form as in (\ref{form}) with discriminant $-108k$, and $x_0, y_0$ are integers for which $F(x_0,y_0)=1$, we can associate an integral point on the Mordell curve $Y^2=X^3+k$.
   
    Conversely suppose, for a given integer $k$, that $(X,Y)$ satisfies equation (\ref{Mord}). To the pair $(X,Y)$, we associate the cubic form 
$$
F(x,y)=x^3-3X xy^2+2Yy^3. 
$$
Such a form $F$ is of the shape (\ref{form}), with discriminant  
$$
D_F=-108Y^2+108X^3=-108k
$$
and covariants satisfying 
$$
X= \frac{G_1(1,0)}{2}=\frac{G(1,0)}{54} \; \mbox{ and } \;  Y= H_1(1,0)= \frac{H(1,0)}{9}.
$$

 In summary, there exists a correspondence between the set of integral solutions 
 $$
 S_k = \left\{ (X_1,Y_1), \ldots, (X_{N_k},Y_{N_k}) \right\}
 $$
to the Mordell equation  $Y^2=X^3+k$ and the set $T_k$ of triples  $(F,x,y)$, where each $F$ is a binary cubic form of the shape (\ref{form}), with discriminant $-108k,$ and the integers $x$ and $y$ satisfy $F(x,y)=1$. Note that the forms $F$ under consideration here need not be irreducible.
 
In the remainder of this section, we will  show that there is, in fact,  a bijection between $T_k$ under $SL_2(\mathbb{Z})$-equivalence, and the set $S_k$. We begin by demonstrating the following pair of lemmata.
\begin{lemma} Let $k$ be a nonzero integer and suppose that
$F_1$ and $F_2$ are $SL_2 ( \mathbb{Z})$-inequivalent binary cubic forms of the shape (\ref{form}), each with discriminant $-108k$, and that $x_1, y_1, x_2$ and $y_2$ are integers such that  $(F_1,x_1,y_1)$ and $(F_2,x_2,y_2)$ are in $T_k$.
Then the tuples $(F_1,x_1,y_1)$ and $(F_2,x_2,y_2)$ correspond to distinct elements of $S_k$.
\end{lemma}
 \begin{proof}
Suppose that the tuples $(F_1,x_1,y_1)$ and $(F_2,x_2,y_2)$ are in $T_k$, so that
$$
F_1(x_1,y_1)=F_2(x_2,y_2)=1.
$$
Since, for each $i$,  $x_i$ and $y_i$ are necessarily coprime, we can find integers $m_i$ and $n_i$ such that $m_ix_i-n_iy_i=1$, for $i=1, 2$. Writing $\tau_i = \left(\begin{smallmatrix} 
x_i & n_i \\
y_i & m_i
 \end{smallmatrix}\right)$, we thus have
 $$
 F_i \circ \tau_i  (x,y) = x^3+3b_ix^2y+3c_ixy^2+d_iy^3,
 $$
 for integers $b_i, c_i$ and $d_i$, and hence, under the further action of 
 $\gamma_i = \left(\begin{smallmatrix} 
  1 & -b_i \\
  1 & 0 
 \end{smallmatrix}\right),$
 we observe that $F_i$ is $SL_2 ( \mathbb{Z})$-equivalent to 
 $$
 x^3-3p_ixy^2+2q_iy^3,
 $$
 where
 $$
p_i=\frac{G_{F_i}(x_i,y_i)}{54} \  \  \text{and} \  \ q_i = \frac{H_{F_i}(x_i,y_i)}{9}.
 $$
  If the two tuples correspond to the same element of $S_k$, necessarily
 $$
 G_{F_1}(x_1,y_1) = G_{F_2}(x_2,y_2) \; \mbox{ and } \;  H_{F_1}(x_1,y_1)  = H_{F_2}(x_2,y_2),
 $$
contradicting our assumption that $F_1$ and $F_2$ are $SL_2 ( \mathbb{Z})$-inequivalent.
 \end{proof}
      
\begin{lemma}
Suppose that $k$ is a nonzero integer, that $F$ is a binary cubic form of the shape (\ref{form}) and discriminant $-108k$, and that $F(x_0,y_0)=F(x_1,y_1)=1$ where $(x_0,y_0)$ and $(x_1,y_1)$ are distinct pairs of integers. Then the tuples $(F,x_0,y_0)$ and $(F,x_1,y_1)$ correspond to distinct elements of $S_k$.
\end{lemma}
   
\begin{proof}
Via $SL_2 ( \mathbb{Z})$-action, we may suppose, without loss of generality, that $F(x,y) = x^3+3bx^2y+3cxy^2+dy^3$ and that $(x_0,y_0)=(1,0)$. If the triples
$(F,1,0)$ and $(F,x_1,y_1)$ correspond to the same element of $S_k$, necessarily
$$
 G_{F}(1,0) = G_{F}(x_1,y_1) \; \mbox{ and } \;  H_{F}(1,0)  = H_{F}(x_1,y_1),
 $$
 whereby
 \begin{equation} \label{temp1}
 x_1^3+3bx_1^2y_1+3cx_1y_1^2+dy_1^3=1,
 \end{equation}
 \begin{equation} \label{temp2}
 (b^2-c) x_1^2 + (bc-d) x_1y_1 + (c^2-bd) y_1^2 = (b^2-c)
 \end{equation}
and
\begin{equation} \label{temp3}
a_1 x_1^3 + 3 b_1 x_1^2 y_1 + 3 c_1 x_1 y_1^2 + d_1 y_1^3 = a_1.
 \end{equation}
It follows that
$$
3 (b a_1-b_1) x_1^2 y_1 + 3 (c a_1 - c_1) x_1 y_1^2 + (d a_1 - d_1) y_1^3=0.
$$
Since $(x_1,y_1) \neq (1,0)$, we have that $y_1 \neq 0$ and so
$$
3 (b a_1-b_1) x_1^2 + 3 (c a_1 - c_1) x_1 y_1 + (d a_1 - d_1) y_1^2=0,
$$
i.e.
\begin{equation} \label{temp4}
-3 (b^2-c)^2 x_1^2 + 3 (b^2-c)(d-bc)  x_1 y_1 + (3bcd-b^3d -c^3-d^2) y_1^2=0.
\end{equation}
If $b^2=c$, it follows that 
$$
3bcd-b^3d -c^3-d^2=- (d-b^3)^2 =0,
$$
so that $d=b^3$ and $F(x,y) = (x+by)^3$, contradicting our assumption that $D_F \neq 0$. We thus have
$$
(b^2-c) x_1^2 + (bc-d) x_1y_1 = \frac{(3bcd-b^3d -c^3-d^2)}{3 ( b^2-c)} y_1^2
$$
and so
$$
\frac{(3bcd-b^3d -c^3-d^2)}{3 ( b^2-c)} y_1^2 = (bd - c^2) y_1^2 + b^2-c,
$$
i.e.
$$
\left(-d^2+6bcd-4b^3d-4c^3+3b^2c^2 \right) y_1^2 = 3 (b^2-c)^2
$$
and so
$$
D_F \, y_1^2 = 81 (b^2-c)^2.
$$
Since $D_F=-108k$, it follows that $k = -3 m^2$ for some integer $m$, where $2 m y_1 = b^2-c$. 
From (\ref{temp4}), we thus have
$$
-12 m^2 x_1^2  + 6 (d-bc) m x_1  + 3bcd-b^3d -c^3-d^2 =0,
$$
whereby
$$
x_1 = \frac{1}{4 m } \left( d-bc   \pm  2  \right).
$$
Substituting the expressions for $x_1$ and $y_1$ into equation (\ref{temp4}), we conclude, from $b^2 \neq c$, that
$$
\frac{-3}{4} (d-bc \pm 2)^2 + \frac{3}{2} (d-bc)(d-bc \pm 2) + 3bcd-b^3d -c^3-d^2 =0,
$$
whence, since the left-hand-side of this expression is just $D_F/27-12$, we find that $D_F = 324$. It follows that the triple $(F,1,0)$, say,  corresponds to an integral solution to the Mordell equation $Y^2=X^3-3$. 
Adding $4$ to both sides of this equation, however, we observe that necessarily  $X^2-X+1 \equiv 3 \mod{4}$,  contradicting the fact that it divides the sum of two squares $Y^2+4$. The lemma thus follows  as stated.
\end{proof}

  To conclude as desired, we have only to note that, for any $\gamma \in  SL_2 ( \mathbb{Z}),$ covariance implies that $H_{F \circ \gamma}=H_F \circ \gamma$ and  $G_{F \circ \gamma}=G_F \circ \gamma$, and hence triples
  $$
  (F, x_0, y_0) \; \mbox{ and } \;  (F \circ \gamma, \gamma(x_0), \gamma(y_0))
  $$
  in $T_k$ necessarily correspond to the same solution to (\ref{Mord}) in $S_k$.
  
\begin{rem}
 Instead of working with  $SL_2 ( \mathbb{Z})$-equivalence, we can instead consider $GL_2 ( \mathbb{Z})$-equivalence classes (and, as we shall see in the next section, this equivalence is arguably a more natural one with which to work). Since $H(x,y)$ and $G^2(x,y)$ are $GL_2 ( \mathbb{Z})$-covariant, if two forms are equivalent under the action of $GL_2 ( \mathbb{Z})$, but not under $SL_2 ( \mathbb{Z})$, then we have 
 $$
 H_{F \circ \gamma}=H_F \circ \gamma \ \  \text {and}  \ \  G_{F \circ \gamma}=-G_F \circ \gamma.
 $$
It follows that, in order to determine all pairs of integers $(X,Y)$ satisfying  equation (\ref{Mord}), it is sufficient  to find a representative for each   $GL_2 ( \mathbb{Z})$-equivalence class of forms of shape (\ref{form}) and discriminant $-108k$  and, for each such  form, solve the corresponding Thue equation. A pair of integers  $(x_0,y_0)$ for which $F(x_0,y_0)=1$ now leads to a pair of solutions $(X, \pm Y)$ to $Y^2=X^3+k$, where
 $$
 X=H_1(x_0,y_0) \; \mbox{ and } \; Y=G_1(x_0,y_0)/2,  
 $$
 at least provided $G_1(x_0,y_0) \neq 0$. 
\end{rem}

\section{Finding representative forms} \label{rep}

As we have demonstrated  in the previous section, to solve Mordell's equation for a given integer $k$, it suffices to determine a set of representatives for $SL_2 ( \mathbb{Z})$-equivalence classes (or, if we prefer,  $GL_2 ( \mathbb{Z})$-equivalence classes) of binary cubic forms of the shape (\ref{form}), with discriminant $-108k$, and then solve the corresponding Thue equations $F(x,y)=1$. In this section, we will describe how to find distinguished representatives for equivalence classes of cubic forms with a given discriminant. In all cases, the various notions of {\it reduction} arise from associating to a given cubic form a particular definite quadratic form -- in case of positive discriminant, the Hessian defined earlier works well. In what follows, we will state our definitions of reduction solely in terms of the coefficients of the given cubic form, keeping the role of the associated quadratic form hidden from view.

\subsection{Forms of positive discriminant}

In case of positive discriminant forms (i.e. those corresponding to negative values of $k$), there is a well-developed classical reduction theory, dating back to work of Hermite \cite{Her1}, \cite{Her2} and later applied to great effect by Davenport (see e.g. \cite{Dav}, \cite{Dav2} and \cite{DaHe}). This procedure allows us to determine distinguished {\it reduced} elements within each equivalence class of forms. We can, in fact, apply this reduction procedure to both irreducible and reducible forms; initially we will assume the forms we are treating are irreducible, for reasons which will become apparent. We will follow work of  Belabas \cite{Be} (see also Belabas and Cohen \cite{BeCo} and Cremona \cite{Cr}), in essence a modern treatment and refinement  of Hermite's method.

\begin{dfn} An irreducible binary integral cubic form 
$$
F(x,y) = ax^3 + 3 b x^2 y + 3 c x y^2 + d y^3
$$
of positive discriminant is called {\it reduced} if we have
 \begin{itemize}
 \item $ | bc-ad| \leq b^2 - ac \leq c^2 - bd$,
 \item  $a>0,b \geq 0$, where $d <0$ whenever $b=0$,
 \item if $bc=ad, d<0$,
 \item if $b^2-ac=bc-ad$, $b< \left \lvert a-b \right \rvert$, and
 \item if $b^2-ac=c^2-bd$, $a \leq \left \lvert d \right\rvert$ and $b < \left \lvert c \right\rvert$.   
 \end{itemize}
 \end{dfn}  
  The main value of this notion  of reduction is apparent in the following result (Corollary 3.3 of \cite{Be}).
  \begin{proposition} \label{goodie}
  Any irreducible cubic form of the shape  (\ref{form}) with positive discriminant is $GL_2 ( \mathbb{Z})$-equivalent to a unique reduced one. 
  \end{proposition}
 To determine equivalence classes of reduced cubic forms with bounded discriminant, we will appeal to the following result (immediate from Lemma 3.5 of Belabas \cite{Be}).
 \begin{lemma}
 Let $K$ be a positive real number and 
 $$
 F(x,y)= a x^3 + 3 b x^2 y + 3 c x y^2 + d y^3
 $$
be a reduced form whose discriminant lies in $(0,K]$. Then we have 
$$
1 \leq a  \leq \frac{2K^{1/4}}{3 \sqrt{3}}
 $$
 and
 $$ 
 0 \leq b \leq \frac{a}{2}+ \frac{1}{3} \, \left( \sqrt{K}-\frac{27a^2}{4} \right)^{1/2}.
 $$
 If we denote by  $P_2$ the unique positive real solution of the equation
 $$ 
    -4P_2^3+(3a+6b)^2P_2^2+27a^2K=0,
 $$
  then
  $$ 
  \frac{9b^2-P_2}{9a} \leq c \leq b-a.
  $$
    \end{lemma}

  In practice, to avoid particularly large loops on the coefficients  $a, b, c$ and $d$, we will instead employ a slight  refinement  of this lemma to treat forms with discriminants in the interval $(K_0,K]$ for given positive reals $K_0 < K$. It is easy to check that we have
$$
\frac{9b^2-P_2}{9a} \leq c \leq  \min \left \{ { \frac{3b^2-(a^2 K_0/4)^{1/3}}{3a}},{b-a} \right \}.
$$
To bound $d$, we note that the definition of reduction implies that     
$$ 
\frac{(a+b)c-b^2}{a} \leq d \leq  \frac{(a-b)c-b^2}{a}.
$$
The further assumption that $K_0 < D_F \leq K$ leads us to a quadratic equation in  $d$ which we can solve to determine a second interval for $d$. Intersecting these intervals provides us with (for values of $K$ that are not too large) a reasonable search space for $d$.
    
\subsection{Forms of negative discriminant}

In case of negative discriminant,  we require a different notion of reduction, as the Hessian is no longer a definite form. We will instead, following Belabas \cite{Be}, appeal to an idea of Berwick and Mathews \cite{BeMa}. We take as our definition of a reduced form an alternative characterization due to Belabas (Lemma 4.2 of \cite{Be}).
\begin{dfn}
An irreducible binary integral cubic form 
$$
F(x,y) = ax^3 + 3 b x^2 y + 3 c x y^2 + d y^3
$$
of negative discriminant is called {\it reduced} if we have
 \begin{itemize}
 \item $d^2 - a^2 > 3 ( bd-ac)$,
 \item  $-(a-3b)^2-3ac< 3 (ad-bc) < (a+3b)^2+3ac$,
 \item $a>0 ,  \, b \geq  0 \ \  \text{and} \ \ d >0 \ \ \text{whenever} \ \ b=0$.   
 \end{itemize}
\end{dfn}      
                 
Analogous to Proposition \ref{goodie}, we have, as a consequence of Lemma 4.3 of \cite{Be} :             
   \begin{proposition}Any irreducible cubic form of the shape  (\ref{form}) with negative discriminant  is $GL_2 ( \mathbb{Z})$-equivalent to a unique reduced one. 
  \end{proposition}

To count the number of reduced cubic forms in this case we appeal to Lemma 4.4 of Belabas \cite{Be} :
   \begin{lemma}
  Let $K$ be a positive real number and 
 $$
 F(x,y)= a x^3 + 3 b x^2 y + 3 c x y^2 + d y^3
 $$
be a reduced form whose discriminant lies in $[-K,0)$. Then we have 
$$ 
1 \leq a \leq   \left ( \frac {16K} {27} \right )^{1/4} 
$$
$$ 
0 \leq b  \leq  \frac{a}{2}+ \frac{1}{3} \,  \left( \sqrt{K/3} - \frac {3a^2}{4} \right)^{1/2}
$$
$$ 
1-3b \leq 3c \leq \left(\frac{K}{4a} \right)^{1/3} + \; \;
\left\{
\begin{array}{cc}
3 b^2/a & \mbox{ if } a \geq 2b, \\
3b - 3a/4 & \mbox{ otherwise.}
\end{array}
\right.
$$
\end{lemma}
  
  As in the case of forms of positive discriminant, from a computational viewpoint it is often useful to restrict our attention to forms with discriminant $\Delta$ with $-\Delta \in ( K_0, K]$ for given $0 < K_0 < K$. Also as previously, the loop over $d$ is specified by the inequalities defining reduced forms and by the definition of discriminant.
  
  It is worth noting here that a somewhat different notion of reduction for cubic forms of negative discriminant is described in Cremona \cite{Cr}, based on classical work of Julia \cite{Ju}. Under this definition, one encounters rather shorter loops for the coefficient $a$ -- it appears that this leads to a slight improvement in the expected complexity of this approach (though the number of tuples $(a,b,c,d)$ considered is still linear in $K$).
  
  \subsection{Reducible forms}
  
  Suppose finally that $F$ is a reducible cubic form of discriminant $-108k$, as in (\ref{form}), for which $F(x_0,y_0)=1$ for some pair of integers $x_0$ and $y_0$. Then, under  $SL_2 ( \mathbb{Z})$-action, $F$ is necessarily equivalent to
  \begin{equation} \label{reduced form}
  f(x,y)=x(x^2+3Bxy+3C y^2),
  \end{equation}
  for certain integers $B$ and $C$.
We thus have  
\begin{equation} \label{disc}
D_f=D_F=27C^2(3B^2-4C)
\end{equation}
(so that necessarily $BC \equiv 0 \mod{2}$). Almost immediate from (\ref{disc}), we have

\begin{lemma} Let $K > 0$ be a real number and suppose that
$f$ is a cubic form as in  (\ref{reduced form}). If we have $0 < D_f \leq K$ then 
$$
   -\left( \frac{K}{108} \right)^{1/3}  \leq C \leq \left( \frac{K}{27} \right)^{1/2}, \; \; C \neq 0
   $$
   and
   $$
   \max \left\{ 0,  \left( \frac{4C+1}{3} \right)^{1/2} \right\} \leq B \leq \left( \frac{K+108 C^3}{81 C^2}  \right)^{1/2}.
   $$
If, on the other hand, $-K \leq D_f < 0$ then
$$
  1 \leq C \leq \left( \frac{K}{27} \right)^{1/2}
$$
and
$$
\max \left\{ 0, \left( \frac{-K+108 C^3}{81 C^2}  \right)^{1/2} \right\} \leq B \leq \left( \frac{4C-1}{3} \right)^{1/2}.
$$
\end{lemma}
 
 One technical detail that remains for us, in the case of reducible forms, is that of  identifying $SL_2 ( \mathbb{Z})$-equivalent forms. If we have
$$
f_1(x,y)= x(x^2+3B_1xy+3C_1 y^2) \; \mbox{ and } \; f_2(x,y)=x(x^2+3B_2xy+3C_2 y^2),
$$
with
\begin{equation} \label{death}
 f_1 \circ \tau  (x,y) = f_2(x,y) \; \mbox{ and } \tau(1,0) = (1,0),
 \end{equation}
 where $\tau \in SL_2 ( \mathbb{Z})$,
then necessarily $\tau =\left(\begin{smallmatrix} 
 1 & n \\
  0 & 1 
 \end{smallmatrix} \right),
 $
 for some integer $n$, so that
 $$
 B_2=B_1+n, \; C_2 = C_1 + 2 B_1 n + n^2 \; \mbox{ and } \;
 3 C_1 n+3 B_1 n^2+n^3=0.
 $$
 Assuming $n \neq 0$, it follows that $f_1(x,1)$ factors completely over $\mathbb{Z} [x]$, whereby, in particular, $C_1=3 C_0$ for $C_0 \in \mathbb{Z}$ and $B_1^2-4C_0$ is a perfect square, say
 $B_1^2-4C_0 = D_0^2$ (where $D_0 \neq 0$ if we assume that $D_{f_1} \neq 0$). We may thus write
 $$
 n = \frac{ -3 B_1 \pm 3 D_0}{2},
 $$
 whence there are precisely three pairs $(B_2, C_2)$ satisfying (\ref{death}), namely
 $$
 (B_2, C_2) = (B_1,C_1), \; \left(  \frac{ - B_1 + 3 D_0}{2},  \frac{3}{2} D_0 (D_0-B_1) \right)
  $$
 and
 $$
 \left(  \frac{ - B_1 - 3 D_0}{2},  \frac{3}{2} D_0 (D_0+B_1) \right).
 $$
 
Let us define a notion of reduction for forms of the shape   (\ref{reduced form}) as follows :

  \begin{dfn} A irreducible binary integral cubic form 
$$
F(x,y) = x ( x^2 + 3 b x y + 3 c y^2 ) 
$$
of nonzero discriminant $D_F$ is called {\it reduced} if we have either
 \begin{itemize}
 \item $D_F$ is not the square of an integer, or 
 \item  $D_F$ is the square of an integer, and  $b$ and $c$ are positive. 
 \end{itemize}
 \end{dfn}  
 
 From the preceding discussion, it follows that such reduced forms are unique in their  $SL_2 ( \mathbb{Z})$-class.
   Note that the solutions to the equation
  $$
  F(x,y) = x(x^2+3bxy+3c y^2) = 1
  $$
  are precisely those given by $(x,y)=(1,0)$ and, if $c \mid b$, $(x,y)=(1,-b/c)$.
  
\section{Running the algorithm}  \label{comp}

We implement the  algorithm implicit in the preceding sections for finding the integral solutions to equations of the shape (\ref{Mord}) with $|k|\leq K$, for given $K > 0$. The number of cubic Thue equations  $F(x,y)=1$ which we are required to solve is of order $K$. To handle these equations, we appeal to by now well-known arguments of Tzanakis and de Weger \cite{TW} (which, as noted previously, are based upon lower bounds for linear forms in complex logarithms, together with lattice basis reduction); these are implemented in several computer algebra packages, including Magma and Pari (Sage). We used the former despite concerns over its reliance on closed-source code, primarily due to its stability for longer runs. The main computational bottleneck in this approach is typically that of computing the fundamental units in the corresponding cubic fields; for computations with $K=10^7$, we encountered no difficulties with any of the Thue equations arising (in particular, the fundamental units occurring can be certified without reliance upon the Generalized Riemann Hypothesis). We are unaware of a computational complexity analysis, heuristic or otherwise, of known algorithms for solving Thue equations and hence it is not by any means obvious how timings for this approach should compare to that based upon lower bounds for linear forms in elliptic logarithms, as in \cite{G}. Using a ``stock'' version of Magma, our approach does have the apparent advantage of not crashing for particular values of $k$ in the range under consideration (i.e with $K=10^7$).

\section{Numerical results}  \label{Numer}

The full output for our computation is available at the weblink

\url{http://www.math.ubc.ca/~bennett/BeGa-data.html},

\noindent  which documents the results of a month-long run on a MacBookPro. Realistically, with sufficient perseverance and suitably many machines, one should be able to readily extend the results described here to something like $K=10^{10}$. In the remainder of this section, we will  briefly summarize our data.

\subsection{Number of solutions} 
In what follows, we tabulate the number of curves encountered for which equation (\ref{Mord}) has  a given number $N_k$ of integer solutions, with $|k| \leq 10^7$. We include results for all values of $k$, and also those obtained by restricting to $6$-th power free $k$ (the two most natural restrictions here are, in our opinion,  this one or the restriction to solutions with $\gcd (X,Y)=1$).
In the  range under consideration, the maximum number of solutions encountered for positive values of $k$ was $58$,  corresponding to the case $k=3470400$. For negative values of $k$, the largest number of solutions we found was $66$, for $k=-9754975$.

Regarding the largest value of $N_k$ known (where, to avoid trivialities, we can, for example, consider only $6$-th power free $k$), Noam Elkies kindly provided the following example, found in October of 2009  :
$$
 k = 509142596247656696242225 = 5^2 \cdot 7^3 \cdot 11^2 \cdot 19^2 \cdot 149 \cdot 587 \cdot 15541336441.
 $$
This value of $k$ corresponds to an elliptic curve of 
rank (at least) $12$, with
(at least) $125$ pairs of integral points (i.e. $N_k \geq 250$), with $X$-coordinates ranging from $-79822305$ to $801153865351455$.

\begin{table}[h]
\caption{Number of Mordell Curves with $N_k$ integral  points  for positive values of $k$ with $ 0 < k \leq 10^7$}
\centering
\begin{tabular}{cc|cc|cc}
\hline \hline
$N_k$ &  \# of curves & $N_k$ & \# of curves & $N_k$ & \# of curves  \\ \hline
0 & 8667066 & 12 & 3890 & 32 & 33 \\
1 & 108 & 14 & 2186 & 34 & 18 \\
2 & 1103303 & 16 & 1187 & 36 & 28 \\
3 & 34 & 17 & 1 & 38 & 17 \\
4 & 145142 & 18 & 589 & 40 & 11 \\
5 & 33 & 20 & 347 & 42 & 6 \\
6 & 55518 & 22 & 197 & 44 & 3 \\
7 & 8 & 24 & 148 & 46 & 5 \\
8 & 13595 & 26 & 91 & 48 & 2 \\
9 & 6 & 28 & 63 &56 & 1 \\
10 & 6308 & 30 & 55 & 58 & 1 \\ \hline
\end{tabular}
\label{tab:one}
\end{table}

\begin{table}[H]
\caption{Number of Mordell Curves with $N_k$ integral  points  for negative values of $k$ with $\left \lvert k  \right \rvert \leq 10^7$}
\centering
\begin{tabular}{cc|cc|cc}
\hline \hline
$N_k$ &  \# of curves & $N_k$ & \# of curves & $N_k$ & \# of curves  \\ \hline
0 & 9165396 & 11 & 4 & 32 & 8 \\
1 & 167 & 12 & 1351 & 34 & 8 \\
2 & 729968 & 14 & 655 & 36 & 2 \\
3 & 23 & 16 & 340 & 38 & 1 \\
4 & 67639 & 18 & 238 & 40 & 1 \\
5 & 10 & 20 & 160 & 42 & 1 \\
6 & 23531 & 22 & 107 & 44 & 2 \\
7 & 3 & 24 & 71 & 46 & 2 \\
8 & 7318 & 26 & 37 & 48 & 1 \\
9 & 6 & 28 & 20 & 50 & 2 \\
10 & 2912 & 30 & 15 & 66 & 1 \\ \hline
\end{tabular}
\label{tab:two}
\end{table}

\begin{table}[H]
\caption{Number of Mordell Curves with $N_k$ integral  points for positive  $k \leq 10^7$ $6$-th power free}
\centering
\begin{tabular}{cc|cc|cc}
\hline \hline
$N_k$ &  \# of curves & $N_k$ & \# of curves & $N_k$ & \# of curves  \\ \hline
0 & 8545578  & 12 &3575  & 32 & 29 \\
1 & 79 & 14 & 1998 & 34 & 18  \\
2 &1067023 & 16 & 1055 & 36 & 22 \\
3 & 24 & 17 & 1 & 38 &15  \\
4 & 139090 & 18 & 506 & 40 & 9 \\
5 & 10 & 20 & 278 & 42 & 6 \\
6 & 51721 & 22 & 161 & 44 &1  \\
7 & 2& 24 &112  & 46 & 5 \\
8 &12271 & 26 &76  & 48 &2  \\
9 & 1 & 28 & 58 &56 & 1 \\
10 & 5756 & 30 &44  &  &  \\ \hline
\end{tabular}
\label{tab:three}
\end{table}

\begin{table}[H]
\caption{Number of Mordell Curves with $N_k$ integral  points for negative $k$, $\left \lvert k  \right \rvert \leq 10^7$ $6$-th power free}
\centering
\begin{tabular}{cc|cc|cc}
\hline \hline
$N_k$ &  \# of curves & $N_k$ & \# of curves & $N_k$ & \# of curves  \\ \hline
0 & 9026739  & 11 & 1 & 32 &8  \\
1 &  109 & 12 & 1174 & 34 & 8  \\
2 & 705268  & 14 & 562  & 36 & 2  \\
3 &  12& 16 &291  & 38 & 1 \\
4 & 63685 & 18 &197  & 42 & 1 \\
5 &  5& 20 &138  & 44 &2  \\
6 & 21883 & 22 &96  & 46 & 2 \\
7 & 3 & 24 & 68 & 48 &1  \\
8 &6644  & 26 & 31 & 50 & 1 \\
9 & 3 & 28 &17  & 66 &  1\\
10 & 2561 & 30 &13  &  &  \\ \hline
\end{tabular}
\label{tab:four}
\end{table}

\subsection{Number of solutions by rank}

From a result of Gross and Silverman \cite{GrSi}, if $k$ is a $6$-th power free integer, then the number of  integral solutions $N_k$ to equation (\ref{Mord}) is bounded above by a constant $N(r)$ that depends only on the Mordell-Weil rank over $\mathbb{Q}$ of the corresponding elliptic curve. It is easy to show that we have $N(0)=6$, corresponding to $k=1$. For larger ranks, we have that $N(1) \geq 12$ (where the only example of a rank $1$ curve we know with $N_k=12$ corresponds to $k=100$), $N(2) \geq 26$ (where $N_{225}=26$), $N(3) \geq 46$ (with $N_{1334025}=46$), $N(4) \geq 56$ (with $N_{5472225}=56$), $N(5) \geq 50$ (with $N_{-9257031}=50$) and $N(6) \geq 66$ (where $N_{-9754975}=66$). The techniques of Ingram \cite{Ing} might enable one to prove that indeed one has $N(1)=12$.

\subsection{Hall's conjecture and large integral points}
Sharp upper bounds for the heights of integer solutions to equation (\ref{Mord}) are intimately connected to the $ABC$-conjecture of Masser and Oesterle. In this particular context, we have the following conjecture of Marshall Hall :
\begin{conjecture} (Hall)
Given $\epsilon >0$, there exists a positive constant $C_\epsilon$ so that, if $k$ is a nonzero integer, then the inequality  
$$ 
\left\lvert X \right\rvert < C_\epsilon |k|^{2+\epsilon} 
$$
holds for all solutions in integers $(X,Y)$ to equation (\ref{Mord}).
\end{conjecture}
The original statement of this conjecture,  in \cite{Hall}, actually predicts that a like inequality holds for $\epsilon = 0$. The current thinking is that such a result is unlikely to be true (though it has not been disproved). 

We next  list all the Mordell curves encountered  with Hall Measure $X^{1/2}/|k| $ exceeding $1$; in each case we round this measure to the second decimal place.  In the range under consideration, we found no new examples to supplement those previously known and recorded in Elkies \cite{Elk} (note that the case with $k=-852135$ was omitted from this paper due to a transcription error) and in work of Jim\'enez Calvo,  Herranz and S\'aez \cite{CHS}.

\begin{table}[h]
\caption{Hall's conjecture extrema for $\left \lvert k \right \rvert \leq  10^7$}
\centering
\begin{tabular}{ccc|ccc}
\hline \hline
$k$ & $X$ & $X^{1/2}/|k|$ & $k$ & $X$ & $X^{1/2}/|k|$ \\ \hline
-1641843	& 5853886516781223 & 46.60 & 1 & 2 & 1.41 \\ \hline
1090	 & 28187351& 	4.87 & 14668 & 384242766 & 1.34	\\ \hline
17 & 5234 &   4.26&14857 & 390620082	& 1.33 \\ \hline
 225 & 720114 &    3.77  &  -2767769 & 12438517260105 & 1.27 \\ \hline
 24 & 8158 &  3.76 & 8569 & 110781386  & 1.23  \\ \hline
 -307 & 939787 & 3.16 & -5190544 & 35495694227489 & 1.15 \\ \hline
 -207 & 367806 &  2.93 &  -852135 & 952764389446 & 1.15 \\ \hline
 28024 & 3790689201  & 2.20	 &  11492 & 154319269	 & 1.08  \\ \hline
 117073 &	 65589428378	& 2.19 &  618 & 421351 &  1.05  \\ \hline
 4401169 & 53197086958290	& 1.66 &  297 &   93844 & 1.03  \\ \hline
\end{tabular}
\label{tab:five}
\end{table}

Our final table lists all the values of $k$ in the range under consideration  for which equation (\ref{Mord}) has a sufficiently large solution :

\begin{table}[h]
\caption{Solutions to (\ref{Mord}) with $X > 10^{12}$ for $\left \lvert k \right \rvert \leq  10^7$}
\centering
\begin{tabular}{ccc|ccc}
\hline \hline
$k$ & $N_k$ & $X$ & $k$ & $N_k$ & $X$ \\ \hline
-1641843	& 6	& 5853886516781223 &  -1923767	& 2	& 2434890738626 \\
4401169	& 6	& 53197086958290	& -2860984	& 4	& 2115366915022  \\
-5190544	& 4	& 35495694227489	&  2383593	& 10	& 1854521158546 \\
-4090263	& 2	& 16544006443618	&  2381192	& 10	& 1852119707102 \\
-4203905	& 2	& 15972973971249	& -4024909	& 4	& 1569699004069  \\
-2767769	& 4	& 12438517260105	& -9218431	& 6	& 1183858050175 \\
7008155	& 2	& 9137950007869	&  5066001	& 8	& 1026067837540 \\
-9698283	& 2	& 7067107221619	&  5059001	& 8	& 1024245337460 \\
2214289	& 4	& 4608439927403	& 3537071	& 6	& 1007988055117 \\ \hline
\end{tabular}
\label{tab:six}
\end{table}

In particular, by way of example, the equation 
$$
Y^2= X^3 -4090263
$$
thus has the feature that its smallest (indeed only) solution in positive integers is given by
$$
(X,Y)=(16544006443618, 67291628068556097113).
$$
 

    \end{document}